\documentclass[platex]{amsart}

\usepackage{amsmath,amsthm,amssymb,amscd}
\usepackage{ascmac} 
\usepackage[arrow,matrix]{xy}
\usepackage{enumerate}
\usepackage[dvipdfmx]{graphicx}

\usepackage[english]{babel}
\usepackage{comment,url}

\usepackage[usenames]{color}

\def\Ker{\mathop{\mathrm{Ker}}\nolimits}

\def\GL{\mathop{\mathrm{GL}}\nolimits}

\def\mod{\mathop{\mathrm{mod}}\nolimits}

\newcommand{\mf}[1]{{\mathfrak{#1}}}

\newcommand{\bb}[1]{{\mathbb{#1}}}
\newcommand{\mca}[1]{{\mathcal{#1}}}

\newcommand{\inj}{\hookrightarrow}
\newcommand{\surj}{\twoheadrightarrow}
\newcommand{\act}{\curvearrowright}
\newcommand{\congto}{\overset{\cong}{\to}}

\newcommand{\N}{\bb{N}}
\newcommand{\Z}{\bb{Z}}
\newcommand{\Zp}{\bb{Z}_{p}}
\newcommand{\Q}{\bb{Q}}
\newcommand{\Qp}{\bb{Q}_{p}}

\newcommand{\C}{\bb{C}}
\newcommand{\Cp}{\bb{C}_p}
\newcommand{\F}{\bb{F}}
\newcommand{\Fp}{\bb{F}_p}

\newcommand{\m}{\mf{m}}

\newcommand{\mh}{\textsc{m}}

\newcommand{\ol}{\overline}

\newcommand{\ds}{\displaystyle}

\newcommand{\wh}[1]{{\widehat{#1}}}

\DeclareMathOperator*{\restprod}%
 {\mathchoice{\ooalign{\ensuremath{\displaystyle\prod}\crcr\ensuremath{\displaystyle\coprod}}}%
             {\ooalign{\ensuremath{\textstyle\prod}\crcr\ensuremath{\textstyle\coprod}}}%
             {\ooalign{\ensuremath{\scriptstyle\prod}\crcr\ensuremath{\scriptstyle\coprod}}}%
             {\ooalign{\ensuremath{\scriptscriptstyle\prod}\crcr\ensuremath{\scriptscriptstyle\coprod}}}%
}

\theoremstyle{definition}
\newtheorem{thm}{Theorem}[section]
\newtheorem{crl}[thm]{Corollary}

\newtheorem{lmm}[thm]{Lemma}

\newtheorem{prp}[thm]{Proposition}

\newtheorem{dfn}[thm]{Definition}
\newtheorem{exa}[thm]{Example}
\theoremstyle{remark}
\newtheorem{rem}[thm]{Remark}

\title{$p$-adic Mahler measure and $\Z$-covers of links
}
\author{Jun Ueki}
\date{\today}


\begin{document}
\maketitle  

\begin{abstract} 
Let $p$ be a prime number. We develop a theory of $p$-adic Mahler measure of polynomials and apply it to the study of $\Z$-covers of rational homology 3-spheres branched over links. 
We obtain a $p$-adic analogue of the asymptotic formula of the torsion homology growth and a balance formula among the leading coefficient of the Alexander polynomial, the $p$-adic entropy, and the Iwasawa $\mu_p$-invariant. 
We also apply the purely $p$-adic theory of Besser--Deninger to $\Z$-covers of links. 
In addition, we study the entropies of profinite cyclic covers of links. 
We examine various examples throughout the paper. 
\end{abstract}


{
\tableofcontents }

\section{Introduction} 
It is well known that the asymptotic behavior of the torsion homology growth in a $\Z$-cover of $S^3$ branched over a link is described by the Euclidean Mahler measure of the reduced Alexander polynomial of the branch link. 

The Euclidean Mahler measure $\mh(f(t))$ of a Laurent polynomial $0\neq f(t) \in \C[t^{\pm1}]$ is defined by 
$\ds \log \mh(f(t))=\frac{1}{2\pi\sqrt{-1}}\int _{|z|=1} \log|f(z)|\frac{dz}{z}$. 
If $\ds f(t)=a\prod_i (t-\alpha_i)$, then we have Jensen's formula $\ds \mh(f(t))=|a|\prod_i {\rm max}\{|\alpha_i|,1\}$ (\cite{EverestWard1999}). 

If $H$ is a finite group, then we denote its order by $|H|$. 
If $H$ is an infinite group, then we put $|H|=0$. 
For two polynomials $f(t)$ and $g(t)$ in $\Z[t]$, we denote their resultant by $R(f(t),g(t))$. 
Now let $M$ be a rational homology 3-sphere ($\Q$HS$^3$) and let $L$ be a link in $M$. 
Let $h_{\infty}:X_\infty \to X:=M-L$ be a $\Z$-cover over the exterior with a generator $t$ of the deck transformation group and let $\{h_n:M_n\to M\}_n$ denote the system of branched $\Z/n\Z$-covers over $(M,L)$ obtained from subcovers of $h_\infty$ by the Fox completion. 
The Alexander module $H_1(X_\infty)$ is a finitely generated $\Z[t^\Z]$-module. 
We take a generator  $A_L(t)$ of the maximum principal ideal containing the Fitting ideal ${\rm Fitt}_{\Z[t^\Z]} H_1(X_\infty)$ of $H_1(X_\infty)$ over $\Z[t^{\Z}]$ with $A_L(t)\in \Z[t]\subset \Z[t^\Z]$ 
and call it \emph{an/the Alexander polynomial of} $h_\infty$. 
Put $\nu_n(t):=(t^n-1)/(t-1)$. 
Then by a generalization of Fox's formula (Proposition \ref{Fox}, e.g. \cite{KM2008}), there exists some finite set $\mathcal{C}\subset \N$ such that for any $n\gg0$ the resultant $R(A_L(t), \nu_n(t))$ satisfies 
$|H_1(M_n)|=|R(A_L(t),\nu_n(t))|\times C$ for some $C\in\mca{C}$. 
By a similar argument to \cite{GS1991}, we obtain an asymptotic formula of the homology growth (Theorem \ref{pGS}):  
$\ds \lim_{n\in \N;\text{value}\neq 0} |H_1(M_n)|^{1/n}=\mh(A_L(t))$. 

For a prime number $p$, let $|\cdot|_p$ denote the $p$-adic norm (absolute value) normalized by $|p|_p=p^{-1}$. 
If $H$ is a group, then we write $||H||_p=|(|H|)|_p$. 
For an inverse system $\{M_{p^r}\to M\}_r$ of branched $\Z/p^r\Z$-covers over $(M,L)$, a $p$-adic refinement of the asymptotic formula called \emph{the Iwasawa type formula} (Proposition \ref{Iwasawa}) is known: 
If $M_{p^r}$'s are $\Q$HS$^3$'s, then there exist some $\lambda_p,\mu_p \in \N, \nu_p\in \Z$ called the Iwasawa invariants such that for any $r\gg 0$ the equation $||H_1(M_{p^r})||_p^{-1}=p^{\lambda_p r +\mu_p p^r+\nu_p}$ holds. 

The main purpose of this paper is to investigate the relation between the Mahler measure and the Iwasawa invariants. \\

For a prime number $p$, let $\Cp=\wh{\ol{\Q}}_p$ denote the $p$-adic completion of an algebraic closure of the $p$-adic number field $\Qp$, and fix an immersion $\ol{\Q}\inj \Cp$ of an algebraic closure of $\Q$. (We regard $p=\infty$ for Euclidean objects.)  

We define the $p$-adic Mahler measure $\mh_p(f(z))$ of a Laurent polynomial $f(t)\in \Cp[t^{\pm1}]$ ($\neq 0$) by imitating the Shnirel'man integral (\cite{Schnirelmann1938}) as follows: 
We first suppose that $f(t)$ does not vanish on the $p$-adic unit circle $|z|_p=1$ and 
put 
$$\log \mh_p(f(z))
:=\lim_{n\in \N; {\rm gcd}(n,p)=1}\frac{1}{n}\sum_{\zeta^n=1}\log |f(\zeta)|_p,$$
where the limit is taken with respect to the Euclidean topology. 
If $\ds f(t)=\sum_{0\leq i\leq l} a_it^{d-i}=a_0 t^{d-l} \prod_{0\leq i\leq l}(t-\alpha_i)$ with $d\in \Z$ and $l\in \N$, then an analogue $\ds \mh_p(f(t))=|a_0|_p \prod_{0\leq i\leq l} {\rm max}\{|\alpha_i|_p,1\}$ of Jensen's formula holds (Theorem \ref{pJensen}). In addition, the equation $\mh_p(f(t))={\rm max}\{|a_i|_p\}_i$ holds (Proposition \ref{Mp-cal}), that is, 
$\mh_p(f(t))$ coincides with the Gauss norm of $f(t)$. 

We extend this notion with Jensen's formula to the case with roots on $|z|_p=1$ and 
remove the condition ${\rm gcd}(n,p)=1$ on the limit (Theorem \ref{Mp-ext1}). 
Then for a $\Z$-cover over $(M,L)$ with its polynomial $A_L(t)$, we obtain a $p$-adic analogue 
$\ds \lim_{n\in \N;\text{value}\neq 0} ||H_1(M_n)||_p^{1/n}=\mh_p(A_L(t))$ 
of the asymptotic formula (Theorem \ref{pGS}). 
In addition, 
the formula $\log \mh_p(A_L(t))=-\mu_p \log p$ of the Iwasawa invariant follows 
(Proposition \ref{A-mu}).\\ 

Now we consider a $\Z$-cover over $(S^3,L)$ with $\ds A_L(t)=\sum_{0\leq i\leq l} a_it^{d-i}=a_0 t^{d-l}\prod_{0\leq i\leq l}(t-\alpha_i)$ over $\ol{\Q}$. 
The relation between the Mahler measure and the topological entropy is well known. 
It is also known that $a_0$ has a topological information of $L$. 
Noguchi studied the relation among them from a viewpoint of $p$-adic topology (\cite{Noguchi2007}). 
We generalize his work for links. 

We consider the meridian action on $H_1(X_\infty,\Q)$ and its Pontryagin dual. 
Then by \cite{BrunoVirili2015} and \cite{LindWard1988}, the algebraic/topological entropy is given by $\ds h=\sum_{p\leq \infty}h_p=$ $\ds \sum_{p\leq \infty} \log \mh_p(A_L(t))$ with use of the $p$-adic entropies \ $\ds h_p=\log \mh_p(A_L(t)/a_0)=$ \ $\ds \sum_{|\alpha_i|_p>1}\log |\alpha_i|_p$ (Proposition \ref{Z-entropy}). 
We see a balance formula among $a_0$, $h_p$, and the Iwasawa $\mu_p$-invariant (Proposition \ref{hp}): 
$$-\log |a_0|_p=h_p+\mu_p \log p.$$ 
By the product formula $\ds \prod_{p\leq \infty}|a_0|_p=1$, 
a generalization of \cite[Corollary 4]{Noguchi2007} follows:
$\ds \log |a_0|=\sum_{p<\infty}\sum_{|\alpha_i|_p>1}\log |\alpha_i|_p+\sum_{p<\infty}\mu_p\log p$
(Corollary \ref{a_0}). \\ 

Besser and Deninger (\cite{BesserDeninger1999}) defined the \emph{purely} $p$-adic log Mahler measure $m_p$, which is different from ours, with use of the $p$-adic logarithm and the Shnirel'man integral, and proved Jensen's formula for it. 
In addition, Deninger (\cite{Deninger2009}) introduced the notion of the \emph{purely} $p$-adic entropy $\hbar_p$ by using the numbers of fixed points and the $p$-adic logarithm, and proved the relation between $m_p$ and $\hbar_p$. 
In Section 4, we apply their theory to $\Z$-covers of $S^3$ branched over links. \\ 

Finally, in Section 5, we study profinite cyclic covers of $S^3$ branched over links. 
We consider the $p$-adic integer ring $\Zp=\varprojlim_n \Z/p^n\Z$ and 
the profinite integer ring $\wh{\Z}=\varprojlim_n \Z/n\Z$. 
We investigate the entropies of the meridian actions on the Iwasawa modules $\mca{H}_p$ of $\Zp$-covers 
and those on the completed Alexander modules $\mca{H}$ of $\wh{\Z}$-covers for the cases with $A_L(t)\in \Z[t]$.\\ 
\ \ \ \ In Sections 3--5, we compute various examples obtained from links in the Rolfsen table (\cite{Rolfsen1976}).

\section{
$p$-adic Mahler measure of polynomials} 
\subsection{Euclidean Mahler measure}
We consider the complex function on $\C^*$ given by a Laurent polynomial $0\neq f(t)\in \C[t^{\pm1}]$. 
If $\ds f(t)=\sum_{0\leq i\leq l}a_it^{d-i}=a_0 t^{d-l}\prod_{0\leq i\leq l}(t-\alpha_i)$ with $a_0a_l\neq 0$ and it has no root on the unit circle $|z|=1$, then 
its Euclidean Mahler measure $\mh(f(t))$ is defined by 
$$\log \mh(f(t))=\frac{1}{2\pi\sqrt{-1}}\int_{|z|=1}\log|f(z)|\frac{d z}{z}.$$
It satisfies Jensen's formula 
\begin{align*} 
\log \mh(f(t))
&=\log|a_l|-\sum_{0<|\alpha_i|<1} \log |\alpha_i|\\
&=\log |a_0|+\sum_{1<|\alpha_i|}\log |\alpha_i|.
\end{align*} 
It is generalized to the cases with roots on $|z|=1$. 
We define $\mh(f(t))$ as the limit of the integrals along paths avoiding roots on $|z|=1$ \cite{EverestWard1999}. 
Indeed, suppose $f(\alpha)=0, |\alpha|=1$. Then for each simple closed curve $\gamma$ around $\alpha$ with finite length $|\gamma|$, 
the integration of a bounded continuous function $\log |f(z)|/z$ along $\gamma$ is bounded by $|\gamma|\times (\text{constant})$. Thus we have $\ds \lim_{|\gamma|\to 0} \int_{\gamma}\log|f(z)|\frac{dz}{z}=0$. \\ 
  
For polynomials $f(t)=a\prod_i (t-\alpha_i)$, $g(t)=b\prod_j(t-\beta_j) \in \Z[t]$, 
their resultant is defined as $$R(f(t),g(t)):=a^{\deg(g)}b^{\deg (f)}\prod_{i,j}(\alpha_i-\beta_j).$$ 
It coincides with the determinant of the Sylvester matrix, whose elements are given by their coefficients, 
and hence satisfies $R(f(t),g(t))\in \Z$. 
We have $R(f(t),g(t))=0$ if and only if they gave no common zero. 
We have 
\begin{align*}
R(f(t),g(t))&=a^{\deg(g)}\prod_i g(\alpha_i)=b^{\deg(f)}\prod_j f(\beta_j),\\
R(f(t),g(t))&=(-1)^{\deg(fg)}R(g(t),f(t)),\\ 
R(f(t),g(t)h(t))&=R(f(t),g(t))R(f(t),h(t))  
\end{align*}
for each $h(t)\in \Z[t]$. 
If $f(t)=0$, then we put $R(0,g(t))=0$. 
If $g(t)\neq 0$, then we put $R(1,g(t))=1$ (\cite{Weber1979}). \\

By an argument in \cite{GS1991}, 
we have an asymptotic formula of the resultants: 
\begin{prp} \label{asym}
Let $0\neq f(t)\in \Z[t]$ and $\ds \nu_n(t)=(t^n-1)/(t-1)=\sum_{0\leq i<n} t^i$. Then 
$$\lim_{n\in \N;\text{value}\neq 0}|R(f(t),\nu_n(t))|^{1/n}=\mh(f(t)).$$ 
\end{prp} 
In order to prove this, we need to show 
$\ds \lim_{n\in \N;\text{value}\neq 0}|\alpha^n-1|^{1/n}=1$ for a root $\alpha$ with $|\alpha|=1$. 
In \cite{GS1991}, the notion of ${\rm arg}\alpha$ was used in obtaining the estimation $|\alpha^n-1|>C\exp(-(\log n)^6)$ for a constant $C$.

\subsection{$p$-adic Mahler measure {\it $\mh_p$}}
For a prime number $p$, 
let $|\cdot|_p$ denote the $p$-adic norm normalized by $|p|_p=1/p$, and 
let $\Qp$ denote the $p$-adic numbers, that is, the $p$-adic completion of $\Q$. 
Let $\C_p=\wh{\ol{\Q}}_p$ denote the $p$-adic completion of an algebraic closure of $\Qp$, and fix an immersion $\ol{\Q}\inj \C_p$ of an algebraic closure of $\Q$. 
For a continuous function $F:\{z\in \C_p\mid |z|_p=1\}\to \C_p$ with no zero on $|z|_p=1$, 
the Shnirel'man integral is defined by
$$\int_{|z|_p=1}F(z)\frac{d z}{z}:=
\ds \lim_{n\in \N; {\rm gcd}(n,p)=1} \frac{1}{n}\sum_{\zeta^n=1} F(\zeta),$$
where the limit is taken with respect to the $p$-adic topology (\cite{Schnirelmann1938}). 

We \emph{imitate} this notion to define the $p$-adic Mahler measure as follows: 
\begin{dfn} \label{pMahler-def} 
Let $f:\C_p^*\to \C_p$ be the continuous map defined by a Laurent polynomial $0\neq f(t)\in \Cp[t^{\pm1}]$. 
If $f(t)$ has no root on $|z|_p=1$, then we define its \emph{$p$-adic Mahler measure} $\mh_p(f(t))$ by 
$$\log \mh_p(f(t))=\ds \lim_{n\in \N; {\rm gcd}(n,p)=1} \frac{1}{n}\sum_{\zeta^n=1} \log |f(\zeta)|_p,$$
where the limit is taken with respect to the Euclidean topology. 
\end{dfn}

Then we have an analogue of Jensen's formula: 
\begin{thm}[Jensen's formula] \label{pJensen} 
If $\ds f(t)=\sum_{0\leq i\leq l}a_it^{d-i}=a_0 t^{d-l}\prod_{0\leq i\leq l}(t-\alpha_i) \in \Cp[t^{\pm 1}]$ with $a_0a_l\neq 0$, then  
\begin{align*} 
\log \mh_p(f(t))
&=\log |a_l
|_p-\sum_{
0<|\alpha_i|_p<1} \log |\alpha_i|_p\\
&=\log |a_0
|_p+\sum_{
1<|\alpha_i|_p}\log |\alpha_i|_p.
\end{align*} 
\end{thm}

\begin{rem}
Besser and Deninger defined the \emph{purely} $p$-adic Mahler measure with use of the Shnirel'man integral and the $p$-adic log, which is different from ours, and proved an analogue of Jensen's formula for it (\cite{BesserDeninger1999}). 
We recall it in Section 4. 
\end{rem}

For our $\mh_p(f(t))$, we can remove the condition ${\rm gcd}(n,p)=1$ on the limit. 
In addition, we can extend this notion for $f(t)$ with roots on $|z|_p=1$. 
(We discuss another modification in Subsection \ref{sekibunro}.) Namely, 

\begin{thm} \label{Mp-ext1} 
Let $f(t) \in \Cp[t^{\pm 1}]$ ($\neq 0$) which may have roots on $|z|_p=1$. 
If we modify the definition of the $p$-adic Mahler measure (Definition \ref{pJensen}) as 
$$\log \mh_p(f(t))=\lim_{n\in \N} \frac{1}{n} \sum_{\zeta^n=1; f(\zeta)\neq 0} \log |f(\zeta)|_p,$$
then Jensen's formula (Theorem \ref{pJensen}) still holds. 
\end{thm}

We will prove Theorem \ref{Mp-ext1} in Subsection \ref{ssProofs}.
It is nontrivial if $f(t)$ has a root on $|z|_p=1$ which is not a root of unity. 
(We have $|1+p|_p=1$ for instance.) 
Since there is no analogous notions of $\arg \alpha$ and the geometric trigonometric functions for $p$-adic numbers, 
we cannot directly follow the proof in \cite{GS1991}. 
We need to give an alternative framework using properties of roots of unity and the strong triangle inequality.

By Theorem \ref{Mp-ext1}, we obtain a $p$-adic analogue of the asymptotic formula: 
\begin{prp} \label{asym-p}
If $0\neq f(t)\in \Cp[t]$, then 
$$\lim_{n\in \N;\text{value}\neq 0}|R(f(t),\nu_n(t))|_p^{1/n}=\mh_p(f(t)).$$
\end{prp}

We note that the $p$-adic Mahler measure of a polynomial can be calculated easily from its coefficients (e.g. the argument in the end of \cite{LindWard1988}):  

\begin{prp} \label{Mp-cal}
If $\ds f(t)=\sum_{0\leq i\leq l} a_it^{d-i} \in \Cp[t^{\pm1}]$ with $d\in \Z$ and $l\in \N$, then $\mh_p(f(t))={\rm max}\{|a_i|_p\}_i$ holds. 
In other words, the $p$-adic Mahler measure and the Gauss norm of $f(t)$ coincide.  
\end{prp}

\begin{proof} 
Let the roots of $f(t)$ be indexed by natural numbers as $\alpha_i$ in a descending order with respect to the $p$-adic norm. 
If $r$ is the largest $i$ with $|\alpha_i|_p>1$, then we have 
$|a_r/a_0|_p=|\alpha_1\cdots \alpha_r +\text{smaller terms}|_p=|\alpha_1\cdots \alpha_r|_p$. \end{proof}

\subsection{Proofs} \label{ssProofs}
To begin with, we recall properties of the $p$-adic norm $|\cdot|_p$. 
It defines an ultra-metric, that is, it satisfies the strong triangle inequality 
$|x-y|_p\leq {\rm max}\{|x|_p,|y|_p\}$ for every $x,y \in \C_p$. 
As a general property of ultra-metric spaces, we have the Krull sharpening: 
If $|x|_p\neq |y|_p$, then $|x-y|_p={\rm max}\{|x|_p,|y|_p\}$ holds. \\ 

Next, we recall several properties of numbers on the $p$-adic unit circle. 

\begin{lmm} \label{alpham}
If $\alpha \in \C_p$ satisfies $|\alpha|_p=1$ and $|\alpha-1|_p<1$, and if $m$ is prime to $p$, then $|\alpha^m-1|_p=|\alpha-1|_p$ holds. 
\end{lmm}

\begin{proof}
Indeed, we have 
$|\alpha^m-1|_p=|\alpha-1|_p|\alpha^{m-1}+\cdots+1|_p=|\alpha-1|_p|(\alpha^{m-1}-1)+\cdots+(\alpha-1)+m|_p=|\alpha-1|_p |(\alpha-1)(\alpha^{m-2}+\cdots +1)+m|_p$. 
Since 
$|\alpha-1|_p<1$, $|\alpha^{m-2}+\cdots +1|_p\leq {\rm max}\{|\alpha^i|_p\}_i=1$, and $|m|_p=1$, 
we have $|(\alpha-1)(\alpha^{m-2}+\cdots +1)+m|_p=1$. 
\end{proof} 

\begin{lmm} \label{1-zeta}
Let $\zeta$ be a primitive $n$-th root of unity and suppose $n=mp^r$ with $p \not| m$. 
If $n=1$, then $|1-\zeta|_p=0$. 
If $m=1$ and $r>0$, that is, if $\zeta$ is a primitive $p$-power-th root of unity with $\zeta\neq 1$, then $|1-\zeta|_p=p^{-1/(p-1)p^{r-1}}$ holds. 
If otherwise, then $|1-\zeta|_p=1$. 
\end{lmm}

\begin{proof}
We have $\ds \prod_{\zeta^n=1,\zeta\neq1} |1-\zeta|_p=|R(\nu_n(t),t-1)|_p=|\nu_n(1)|_p=|n|_p=|p^r|_p=p^{-r}$ and 
$|1-\zeta|_p\leq 1$. 
Since the set of primitive $p^r$-th roots of unity is closed by non-$p$ powers, the assertion follows from Lemma \ref{alpham}.
\end{proof}

\begin{lmm} \label{nearest}
If $\alpha\in \C_p$ satisfies $|\alpha|_p=1$, then there exist a unique $N$ with $p\not | N$ and a unique primitive $N$-th root $\xi$ of unity satisfying $|\alpha-\xi|_p<1$. 
\end{lmm}

\begin{proof}
We consider the valuation ring $\mca{O}_p=\{z\in \C_p\mid |z|_p\leq 1\}$ and its maximal ideal $\m_p=\cup_{n\in \N_{>0}} (p^{1/n})$. 
Since $|\alpha|_p=1$, we have $\alpha\in \mca{O}_p$. 
We consider $\mod \m_p: \mca{O}_p\surj \ol{\F}_p$. 
Elements of $\ol{\F}_p^{\times}$ are non-$p$-power-th roots of unity, and the restriction map 
$\mod \m_p: \{z\in \C_p\mid \text{non-$p$-power-th roots of unity}\} \subset \mca{O}_p \to \ol{\F}_p^{\times}$ is a bijection. 

Thus there exists a unique non-$p$-power-th root of unity $\xi\in \C_p$ satisfying $\xi \mod \m_p=\alpha \mod \m_p$, 
namely, $|\alpha-\xi|_p<1$. 
(This $\xi$ is the image of $\alpha \mod \m_p$ under the unique Teichm\"uller lift $\omega:\ol{\F}_p\inj \C_p$.) 
\end{proof}

\begin{lmm} \label{lim=1}
If $\alpha\in \C_p$ satisfies $|\alpha|_p=1$ and $\alpha\neq 1$, then $\ds \lim_{n\in \N;\alpha^n\neq1}|\alpha^n-1|_p^{1/n}=1$ holds. 
\end{lmm} 

\begin{proof}
By Lemma \ref{nearest}, we have a unique non-$p$-power-th root $\xi$ of unity with $|\alpha-\xi|_p<1$. 
Suppose that $\xi$ is a primitive $N$-th root of unity. Then $N\neq 1$. 
Suppose $\alpha^n\neq1$. Note that 
$|\alpha^n-1|_p=\prod_{\zeta^n=1}|\alpha-\zeta|_p=\prod_{\zeta^n=1}|(\alpha-\xi)+(\xi-\zeta)|_p$ holds. 

If $N\not | n$ and $\zeta^n=1$, then $|\xi-\zeta|_p=|1-\zeta/\xi|_p=1$. Hence for any $n \in \N$ with $\alpha^n \neq 1$, we have $|\alpha^n-1|=1$. 
If $N|n$, then 
we have $\{|\xi-\zeta|_p \mid \zeta^n=1\}=\{|1-\zeta|_p \mid \zeta^n=1\}$. 
The product of all the non-zero elements of this set is $|n|_p$. 
By Lemma \ref{1-zeta}, 
if $\zeta$ runs through primitive $p^r$-th roots of unity for $r\in \N$, then $|1-\zeta|_p$ increase with respect to $r$ and converges to 1 as $r\to \infty$. 
If $\zeta$ is another root of unity with $\zeta\neq 1$, then 
$|1-\zeta|_p=1$. 
Hence we have $|\alpha-\zeta|_p=|\xi-\zeta|_p$ for almost all the roots $\zeta$ of unity. 

Thus we have $|\alpha^n-1|_p=\prod_{\zeta^n=1}|\alpha-\zeta|_p=\bullet\times |n|_p$, where $\bullet$ lies in a finite set with $\bullet\neq 0$. Since $\ds \lim_{n\in \N}|n|_p^{1/n}=1$, 
we obtain $\ds \lim_{n\in \N;\alpha^n\neq 1}|\alpha^n-1|_p^{1/n}=1$.
\end{proof}

The following lemma plays a key roll in the proof of Theorem \ref{Mp-ext1}: 

\begin{lmm} \label{keylemma} 
For any element $\alpha \in \Cp$, the equality 
$\ds \lim_{n\in \N} \prod_{\zeta^n=1; f(\zeta)\neq 0} |\zeta-\alpha|_p^{1/n}={\rm max}\{|\alpha|_p,1\}$ holds, where $\zeta$ runs through roots of unity in $\C_p$. 
\end{lmm}

\begin{proof} Suppose that $\alpha$ is not a root of unity. 
Then we have 
$$\ds \prod_{\zeta^n=1; f(\zeta)\neq 0} |\zeta-\alpha|_p=\frac{\ds \prod_{\zeta^n=1}|\zeta-\alpha|_p}{\ds \prod_{\zeta^n=1;f(\zeta)=0}|\zeta-\alpha|_p}.$$ 
The denominator lies in a finite set independent of $n$. Hence 
$\text{(denominator)}^{1/n}\to 1$ as $n\to \infty$. 
If $|\alpha|_p\neq 1$, then 
$\text{(numerator)}^{1/n}=|\alpha^n-1|_p^{1/n}={\rm max}\{|\alpha|_p^n,1\}^{1/n}={\rm max}\{|\alpha|_p,1\}$ holds. 
If $|\alpha|_p=1$, then Lemma \ref{lim=1} assures $\ds \lim_n|\alpha^n-1|_p^{1/n}=1$. 

Suppose instead that $\alpha$ is a root of unity. Then we have 
$$\ds \prod_{\zeta^n=1; f(\zeta)\neq 0} |\zeta-\alpha|_p=\frac{\ds \prod_{\zeta^n=1;\zeta\neq \alpha}|\zeta-\alpha|_p}{\ds \prod_{\zeta^n=1;f(\zeta)=0,\zeta\neq \alpha}|\zeta-\alpha|_p}.$$ 
We have $\text{(denominator)}^{1/n}\to 1$ as $n\to \infty$. 
For $n$'s with $\alpha^n\neq 1$,  we have $\text{(numerator)}=|\alpha^n-1|_p$. 
The previous lemma assures $\text{(numerator)}^{1/n}\to 1$ as $n\to \infty$. 
For $n$'s with $\alpha^n=1$, we have 
$\ds \text{(numerator)}=\prod_{\zeta^n=1,\zeta\neq 1}|\zeta-1|_p=|n|_p$. 
Since $1/n\leq|n|_p \leq 1$, we have $\text{(numerator)}^{1/n}\to 1$ as $n\to \infty$. 
\end{proof} \

\begin{proof}[Proof of Theorem \ref{Mp-ext1}] 
Let $\ds f(t)=a_0 t^{d-l}\prod_i(t-\alpha_i) \in \C_p[t^{\pm 1}]$. 
It is sufficient to prove $$\lim_{n\in \N} \prod_{\zeta^n=1; f(\zeta)\neq 0} |f(\zeta)|_p^{1/n}=|a_0|_p\prod_i {\rm max}\{|\alpha_i|_p, 1\}$$
where $\zeta$ runs through roots of unity in $\C_p$. 
Since the both side of this equality is multiplicative with respect to $f$, 
it is sufficient to prove the assertion on each factor $t-\alpha_i$. 
Now the assertion follows from Lemma \ref{keylemma}. 
\end{proof}



\subsection{Alternative modification of the integral path} 
\label{sekibunro}
The Euclidean Mahler measure of $f(t)\in \C[t^{\pm1}]$ satisfying Jensen's formula was defined 
as the limit of the integral along a path $\gamma:S^1\to \C$ 
as $\gamma$ uniformly approaches the unit circle (\cite{EverestWard1999}). 

Here we give an analogous argument for our $p$-adic Mahler measure of $f(t)\in \Cp[t^{\pm 1}]$. 
We consider the image of roots of unity under a function $g:\C_p\to \C_p$ and 
let $g$ approach ${\rm id}$ uniformly with respect to the $p$-adic topology. 
Namely, for each function $g$ with $f(g(\zeta))\neq 0$ for any root $\zeta$ of unity, 
we put 
$\ds \log \mh_{p,g}(f(t))
:=\lim_{n\in \N} \frac{1}{n}\sum_{\zeta^n=1} \log |f(g(\zeta))|_p$, 
where the limit is taken with respect to the Euclidean topology. 
For each $\varepsilon>0$, we consider $g$'s with 
${\rm sup}\{|g(\zeta)-\zeta|_p\mid \zeta \in \text{(roots of unity)}\}<\varepsilon$ 
and let $\varepsilon\to 0$. 
(Such $g$ exists for each $\varepsilon$, e.g., $g(z)=z+\varepsilon/2$.) 
Then we have 

\begin{prp} 
For any function $g$ sufficiently close to ${\rm id}$, 
the limit $\mh_{p,g}(f(t))$ exists and the value is independent of $g$. 
The limit $\ds \lim_{g\to {\rm id}}\mh_p(f(g(t)))$ exists and coincides with $\mh_p(f(t))$
satisfying Jensen's formula. 
\end{prp}

\begin{proof} 
Since $\log|\bullet|_p$ satisfies the logarithm laws, it is sufficient to consider $f(t)=t-\alpha$ for $\alpha\in \C_p$. 
If $|\alpha|_p\neq 1$, then the assertion immediately follows from a property of ultra-metric. 
If $|\alpha|_p=1$, then we have a unique non-$p$-power-th root $\xi$ of unity satisfying $|\alpha-\xi|_p<1$.
Assume ${\rm sup}\{|g(\zeta)-\zeta|_p\mid \zeta \in \text{(roots of unity)}\}<\varepsilon<1$. 
If $\zeta^n=1$, then we have 
$|g(\zeta)-\alpha|_p=|(g(\zeta)-\zeta)+(\zeta-\xi)+(\xi-\alpha)|_p$. 
By a similar argument to Lemma \ref{lim=1}, 
we obtain $\ds \lim_{n\in \N} \prod_{\zeta^n=1} |g(\zeta)-\alpha|_p=1$. 
Therefore we have $\mh_{p,g}(f(t))={\rm max}\{|\alpha|_p, 1\}$. 
\end{proof}

\section{
$p$-adic Mahler measure and $\Z$-covers of links}  

\subsection{$\Z$-covers and {\it $\mh_p$}} 
Let $M$ be a $\Q$HS$^3$ and let $L$ be a $d$-component link in $M$. 
Let $h_\infty: X_\infty\to X=M-L$ be a $\Z$-cover, which we call a \emph{$\Z$-cover over $(M,L)$}, with a generator $t$ of the deck transformation group $t^{\Z}$. 
Then the Alexander polynomial $A_L(t) \in \Z[t]\subset \Z[t^\Z]$ is a generator of the principal ideal of the Fitting ideal ${\rm Fitt}_{\Z[t^\Z]}H_1(X_\infty)$ of the Alexander module over $\Z[t^\Z]$, as explained in Section 1. 

\begin{exa} Suppose that every component $L_i$ of $L$ is null-homologous, and let $t_i\in H_1(X)$ denote the meridian of each $L_i$. Then a standard $\Z$-cover called \emph{the total linking number cover} (\emph{TLN-cover}) \emph{over} $(M,L)$ is defined by the surjective homomorphism $\tau:H_1(X,\Z)\surj \Z;$ $\forall t_i \mapsto 1$. If $A_L(t)\neq 0$, then we have ${\rm Fitt}_{\Z[t^\Z]}H_1(X_\infty)=(A_L(t))$ 
by \cite[Lemma 3.1]{KM2008}. 

If $M=S^3$, then the multivariable Alexander polynomial $\Delta_L(t_1,\ldots,t_d)$, 
the reduced Alexander polynomial $\Delta_L(t)=\Delta_L(t,\ldots,t)$, 
and the Hosokawa polynomial $H_L(t)$ of $L$ are defined. 
For the TLN-cover over $(S^3,L)$, we have 
$$A_L(t)=(t-1)^{d-1}H_L(t)=
\left\{ \begin{array}{ll}
\Delta_L(t) & {\rm if}\ d=1,\\
(t-1)\Delta_L(t) & {\rm if}\ d>1.  
\end{array} \right.  
$$ 
\end{exa}

For each $n\in \N_{>0}$, let $M_n\to M$ denote the branched $\Z/n\Z$-cover obtained as the Fox completion of the $\Z/n\Z$-subcover of $h_\infty$. Then the orders of groups and the cyclic resultants of $A_L(t)$ satisfy the following well-known formula: 
\begin{prp} \label{Fox} Let the notation be as above. Then there exists some finite set $\mca{C}$ of $\N$ such that for any $n\in \N_{>0}$ 
$$|H_1(M_n)|=|R(A_L(t),\nu_n(t))|\times C$$
holds for some $C\in \mca{C}$. 
\end{prp} 
This proposition follows from Sakuma's theorem (\cite[Theorem 3]{Sakuma1981}) or Porti's theorem (\cite[Theorem 1.1]{Porti2004}) by a similar argument to \cite[Proof of Theorem 2.1]{KM2008}. 
It was initially proved for a knot $L$ in $M=S^3$ by Fox (\cite{FoxFDC3}, \cite{Weber1979}) and for a link $L$ in $M=S^3$ by Mayberry--Murasugi (\cite{MM1982}). 

By Propositions \ref{asym}, \ref{asym-p}, and Theorem \ref{Fox}, we obtain the following asymptotic formulae: 
\begin{thm} \label{pGS} Let the notation be as above. 
If $A_L(t)\neq 0$, then 
\begin{align*}
&\lim_{n\in \N;\text{value}\neq 0}|H_1(M_n)|^{1/n}=\mh(A_L(t)),\\
&\lim_{n\in \N;\text{value}\neq 0}||H_1(M_n)||_p^{1/n}=\mh_p(A_L(t)).
\end{align*}
\end{thm}

The Euclidean formula for knots in $S^3$ was proved by \cite{GS1991}, \cite{Riley1990gr}. 
Riley also gave an estimation of $||H_1(M_n)||_p^{-1}$ from above. 

We denote $p=\infty$ for Euclidean objects. 
For $f(t)\in \Z[t^{\pm1}]$, 
the set of $p$-adic Mahler measures $\mh_p(f(t))$ for $p<\infty$ do not tell that for $p=\infty$. 
The set of those for $p\leq \infty$ would have some meaning. 

\begin{exa} \label{ex0} We use the notation in the Rolfsen table (\cite{Rolfsen1976}). 

(1) If $K=4_1$ (figure 8-knot), then the Alexander polynomial $\Delta_K(t)=t^2-3t+1$ satisfies 
$\mh(\Delta_K(t))=\frac{3+\sqrt{5}}{2}$, $\mh_p(\Delta_K(t))=0$ ($p<\infty$). 

(2) If $L=4^2_1$ (Solomon's knot, sigillum Salomonis), then one of the reduced Alexander polynomial $\Delta_L(t,t^{-1})=2$ satisfies $\mh(2)=0,\mh_2(2)=1/2, \mh_p(2)=0$ ($p\neq 2$). 
The same values are obtained for $A_L(t)=2(t-1)$. 
\end{exa} 
\

\begin{rem}($\Z^d$-covers, non-$\Q$HS$^3$-cases, etc.)
We have an asymptotic formula for $\Z^d$-covers with use of the multivariable Alexander polynomials and the Mahler measures. 
In addition, even if $M_n$'s are not necessarily $\Q$HS$^3$, we have an asymptotic formula of the torsion subgroups $H_1(M_n)_{\rm tor}$ with use of the higher Alexander polynomials and the Mahler measures (\cite{SW2002M}, \cite{Le2014}). 

It is noteworthy that sometimes we obtain special values of zeta functions, $L$-functions, and multiple zeta functions as the Mahler measures of multivariable polynomials (\cite{Smyth1981}, \cite{Lalin2003}). 
Furthermore, in several examples the hyperbolic volumes of knot exteriors relate to the Mahler measures of the $A$-polynomials (\cite{Boyd2002Mahler}, \cite{BoydRVD2003}, \cite{Lalin2004}).
An asymptotic formula with use of the hyperbolic volumes is known as the Bergeron--Venkatesh conjecture and Le's theorem (\cite{BV2013}, \cite{Le2014}). 
We expect $p$-adic analogues of them. 
\end{rem} 

Finally we remark that a recent development of such a study for twisted Alexander polynomials of knots is due to Tange (\cite{TangeRyoto1}).

\subsection{Iwasawa $\mu_p$-invariant and {\it $\mh_p$}} 
Let $M$ be a $\Q$HS$^3$, let $L$ be a link in $M$, and let $\{h_{p^r}:M_{p^r}\to M\}_r$ be an inverse system of branched $\Z/p^r\Z$-covers over $(M,L)$. 
(For such an object, see Subsection \ref{Zp-cover} also.) 
Let $t$ denote the topological generator of the inverse limit of the deck transformation groups $\Z/p^r\Z$ corresponding to 1. Then the Iwasawa module $\mca{H}_p:=\varprojlim_r H_1(M_{p^r}, \Zp)$ is a finitely generated module over $\Zp[[t^{\Zp}]]\cong \Zp[[T]]; t\mapsto 1+T$. 
By \cite[Theorem 4.17]{Ueki2}, 
$M_{p^r}$'s are $\Q$HS$^3$'s if and only if its characteristic polynomial in $\Zp[[t^{\Zp}]]$ 
does not vanish at any $p$-power-th root of unity. 

The following formula is an analogue of Iwasawa's class number formula (\cite{Iwasawa1959}), which was initially given by \cite{HMM2006} and generalized by \cite{KM2008} and \cite{Ueki2}: 
\begin{prp}[Iwasawa type formula] \label{Iwasawa} Let the notation be as above. 
If $M_{p^n}$'s are $\Q$HS$^3$'s, then there exist some $\lambda_p,\mu_p \in \N, \nu_p\in \Z$ called the Iwasawa invariants such that for any $r\gg 0$ in $\N$ 
$$||H_1(M_{p^r})||_p^{-1}=p^{\lambda_p r +\mu_p p^r+\nu_p}$$ holds. 
\end{prp} 

Suppose that the $\Z/p^r\Z$-covers are obtained from a $\Z$-cover over $(M,L)$ with the Alexander polynomial $A_L(t)$. 
Comparing with the results in the previous section, we obtain 
\begin{prp} \label{A-mu} 
$$\log \mh_p(A_L(t))=-\mu_p \log p.$$
\end{prp} 
In this sense, the Iwasawa type formula is a $p$-adic refinement of the asymptotic formula (Theorem \ref{pGS}) with use of $\mh_p$. 

\begin{rem} 
We review what we have done for $\Z$-covers of $S^3$ branched over links $L$. 
By the Mayberry--Murasugi formula 
and the definition of $\mh_p$ imitating the Shnirel'man integral, 
if $A_L(t)$ does not vanish on the $p$-adic unit circle, then we have $\ds \lim_{n\in \N;\gcd(p,n)=1}||H_1(M_n)||_p=\mh_p(A_L(t))$. 
By the Iwasawa type formula, 
if $A_L(t)$ does not vanish at any $p$-power-th root of unity, then we have $\ds \lim_{n=p^r}||H_1(M_n)||_p=p^{-\mu_p}$. 
Note that the intersection of their domains of $n$ is empty. 
By Jensen's formula, the formula $\mh_p(A_L(t))={\rm max}\{|\text{coefficients}|_p\}$, and the $p$-adic Weierstrass preparation theorem, if $A_L(t)$ has no root on $|z|_p=1$, then $\log \mh_p=-\mu_p \log p$ immediately follows. 
We removed the assumption on $A_L(t)$ by extending the definition of $\mh_p$. 
\end{rem}

We have $\mh_p(A_L(t))={\rm max}\{|\text{coefficients}|_p\}$. 
If $L=K$ is a knot, then $A_L(t)=\Delta_K(t)$ satisfies $\Delta_K(1)=\pm1$ and its is primitive. 
Thus for each $n<\infty$, we have $\mh_p(A_L(t))=1$ and $\mu_p=0$. 
In regard to $d\geq2$-component link $L$, various examples with $\mh_p(A_L(t))<1$ and $\mu_p>0$ are given (\cite{KM2008}, \cite{KM2013},  \cite{Ueki3}). 

\begin{exa} \label{ex1} We list up all the examples of TLN-covers with non-trivial $\mu_p$ 
obtained from 2-component links $L$ in the Rolfsen table (\cite{Rolfsen1976}).  
We have $A_L(t)=(t-1)\Delta_L(t^{\epsilon_1},t^{\epsilon_2}), \epsilon_i\in\{\pm1\}$. 
We denote $f\dot{=}g$ if $f$ and $g$ coincide up to multiplication by units of $\Z[t^{\pm1}]$. 

(1) $L=4^2_1$. $\Delta_L(x,y)=1+xy$, $\Delta_L(t,t^{-1})=2$, 
$\mu_2=1$. 

(2) $L=6^2_1$. $\Delta_L(x,y)=1+xy+x^2y^2$, $\Delta_L(t,t^{-1})=3$, 
$\mu_3=1$. 

(3) $L=6^2_3$. $\Delta_L(x,y)=2-x-y+2xy$, $\Delta_L(t,t)=2(t^2-t+1)$, 
$\mu_2=1$. 

(4) $L=7^2_3$. $\Delta_L(x,y)=1-x-y+xy$, $\Delta_L(t,t)\,$ $\dot{=}\,\Delta_L(t,t^{-1})\,\dot{=}\,4(t^2-1)^2$, 
$\mu_2=2$. 
\end{exa}

\subsection{$p$-adic entropy $h_p$ and {\it $\mh_p$}} 
Noguchi studied $\Z$-covers of $S^3$ branched over knots (\cite{Noguchi2007}). 
Let $t$ denote the meridian action on the ($\Q$-)Alexander module $H_1(X_\infty,\Q)\cong \Q^l$, 
and let $\wh{t}$ denote the solenoidal system obtained as the Pontrjagin dual of $t$. 
He applied the result of Lind and Ward (\cite{LindWard1988}) to $\wh{t}$ and calculated the topological entropy $h_{\rm top}$ of $\wh{t}$, which coincides with Bowen's measure theoretic entropy. 

Giordano-Bruno and Virili (\cite{BrunoVirili2015}) studied the algebraic entropies $h_{\rm alg}$ of dynamical systems of $\Q$-linear spaces. They gave a formula parallel to that of \cite{LindWard1988}, so that 
$h_{\rm alg}$ of $t$ coincides with $h_{\rm top}$ of $\wh{t}$ above. 
In what follows, we just call them the entropy. 

These entropies have their origin in the Ergodic theory. They were introduced by Adler, Konheim, and McAndrew (\cite{AKM1965}). Various generalization and duality have been studied (\cite{DikranjanBruno2014} for instance).\\ 

Results of \cite{LindWard1988} and \cite{BrunoVirili2015} 
called the Yuzvinski formula or the Kolmogorov--Sinai formula 
can be stated with use of our $p$-adic Mahler measure as follows: 
\begin{prp} \label{Yuzvinski}
Let $\varphi \in \GL(l,\Q)$ with $l \in \N$ and let $B(t)\in \Q[t]$ denote its (monic) characteristic polynomial. 
Then the algebraic entropy of $\varphi \act \Q^l$ and the topological entropy of its Pontrjagin dual are given by 
$$
h=\sum_{p\leq\infty}h_p,\ \ \ h_p=\log \mh_p(B(t)),$$ 
where $h_p$ denotes the $p$-adic entropy of $\varphi\act \Q^l\otimes \Qp\cong \Qp^l$ or its Pontrjagin dual for $p\leq \infty$.
\end{prp}

In addition, suppose $\ds B(t)=\sum_{0\leq i\leq l} b_it^{d-i}\in\Q[t]$ with $d,l\in \N$. 
Then $\ds \sum_{p<\infty} h_p=\log s$ hols for $s={\rm lcm} \{b_i\}_i$, and 
the primitive polynomial $sB(t)\in \Z[t]$ satisfies $h=\log \mh(sB(t))$.\\ 

Now let $L$ be a link in $S^3$ and let $h_\infty: X_\infty\to X=S^3-L$ be a TLN-$\Z$-cover 
with the Alexander polynomial $\ds A_L(t)=\sum_{0\leq i\leq l} a_it^{d-i}\in \Z[t]$ with $a_0a_l\neq 0$. 
Then we have $H_1(X_\infty;\Q)\cong \Q^l$ as linear spaces over $\Q$ 
and the characteristic polynomial of the meridian action on $H_1(X_\infty;\Q)$ is given by the associated monic polynomial $A_L(t)/a_0$. Hence we have 
\begin{prp} \label{Z-entropy}  
The entropy of the $t$ action on $H_1(X_\infty;\Q)$ is given by 
$$h=\sum_{p\leq\infty}h_p,\ \ \ h_p=\log \mh_p(A_L(t)/a_0).$$ 
\end{prp} 
By the product formula $\ds \prod_{p\leq \infty}|a_0|_p=1$, we also have 
$\ds h=\sum_{p\leq \infty} \log \mh_p(A_L(t))$. 
In addition, the primitive polynomial $A_L(t)/{\rm gcd}\{a_i\}_i\in \Z[t]$ satisfies the equality 
$h=\mh(A_L(t)/{\rm gcd}\{a_i\}_i)$.\\ 

If $L$ is a knot, then $A_L(t)=\Delta_L(t)$ is primitive, and similar results hold for the Alexander module $H_1(X_\infty)$ over $\Z$ (\cite[Theorem 1]{Noguchi2007}). 
If $L$ is a general link, then $A_L(t)$ is not necessarily primitive. The entropy $h_\Z$ of the meridian action on $H_1(X_\infty)$ over $\Z$ is given by $h_\Z=\log \mh (A_L(t))$ and satisfies $h_\Z=h+\log {\rm gcd}\{a_i\}_i$. 

\begin{exa} \label{ex2}
In the Rolfsen table (\cite{Rolfsen1976}), the first example with nontrivial entropy $h$ is $K=4_1$ (figure 8-knot). We have $\Delta_K(t)=t^2-3t+1$, $h=\log \mh_{\infty}(\Delta_K(t))=\log\frac{3+\sqrt{5}}{2}$. 
We list up all the 2-component links $L$ with nontrivial $h$ in the table: 

(1) $L=6^2_2$. $\Delta_L(x,y)=x+y-xy+x^2y+xy^2$, $\Delta_L(t,t)\,\dot{=}\,\Delta_L(t,t^{-1})\,\dot{=}\,2t^2-t+2$, 
$h=h_2=\log 2$. 

(2) $L=6^2_3$. $\Delta_L(x,y)=2-x-y+2xy$, $\Delta_L(t,t^{-1})\,\dot{=}\,t^2-4t+1$, $h=h_{\infty}=\log(2+\sqrt{3})$. 

(3) $L=7^2_1$. $\Delta_L(x,y)=1-x-y+xy-x^2y-xy^2+x^2y^2$. (i) $\Delta_L(t,t)=t^4-2t^3+t^2-2t+1$, 
$h=h_\infty=\log\frac{1+\sqrt{2}+\sqrt{2\sqrt{2}-1}}{2}$. 
(ii) $\Delta_L(t,t^{-1})\,\dot{=}\,2t^2-3t+2$, $h=h_2=\log 2$. 

(4) $L=7^2_2$. $\Delta_L(x,y)=1-x-y+3xy-x^2y-xy^2+x^2y^2$, $\Delta_L(t,t^{-1})\,\dot{=}\,2t^2-5t+2=2(t-2)(t-\frac{1}{2})$, $h_\infty=2$, $h_2=\log |1/2|_2=\log 2$, $h=h_\infty+h_2=\log \mh(2t^2-5t+2)=2\log 2$. 
\end{exa}
\subsection{$h_p$, $\mu_p$, and the leading coefficient $a_0$} 
We continue to study a TLN-$\Z$-cover of $(S^3,L)$. 
The leading coefficient $a_0$ of $A_L(t)$ has geometric information.  
For instance, if $L$ is fibered, then $a_0=1$ holds (e.g., \cite[Theorem 5.12]{Hillman2}). 

Now let $p$ be a prime number again. 
Then $h_p=\log \mh_p(A_L(t)/a_0)$ measures the ratio between the monic and primitive polynomials associated to $A_L(t)$, while  
$\mh_p(A_L(t))=p^{-\mu_p}$ measures the ratio between $A_L(t)$ itself and the primitive one. 

Since $\mh_p(A_L(t)/a_0)=\mh_p(A_L(t))/|a_0|_p$, we obtain a balance formula among $h_p,\mu_p,$ and $a_0$:

\begin{prp} \label{hp} 
$$-\log |a_0|_p=h_p+\mu_p \log p. \phantom{\ds \sum_i}$$
\end{prp} 

\begin{exa} \label{ex3}
We have only one example of 2-component link whose $h_p,\mu_p$ are both nontrivial in the Rolfsen table (\cite{Rolfsen1976}). It is $L=9^2_{23}$. We have 
$\Delta_L(x,y)=x^2-2xy+y^2x^3-4x^2y-4xy-y^3+x^3y+2x^2y^2+xy^3$, $\Delta_L(t,t)=4t^2-10t+4=2(2t^2-5t+2)=2(2-t)(1-2t)$. 
$|a_0|_2=1/4$. 
$h_\infty=\log 2$, $h_2=\log |1/2|_2=\log 2$, $h=h_\infty+h_2=\log \mh_\infty(2t^2-5t+2)=2\log 2$. 
$\mh_2(\Delta_L(t,t))=|2|_2=1/2$, $\mu_2=1$. 
\end{exa} 

If $A_L(t)=a_0\prod_i(t-\alpha_i)$, then by Theorems \ref{pJensen}, \ref{Mp-ext1}, and Proposition \ref{Z-entropy}, we have $\ds h_p=\sum_{|\alpha_i|_p>1}\log |\alpha_i|_p$. Note the product formula $\ds \ds \prod_{p<\infty}|a_0|_p=1/|a_0|$. 
Then the previous proposition gives 

\begin{crl} \label{a_0}
$$\ds \log |a_0|=\sum_{p<\infty}\sum_{|\alpha_i|_p>1}\log |\alpha_i|_p+\sum_{p<\infty}\mu_p\log p. $$
\end{crl}

This formula generalizes \cite[Corollary 4]{Noguchi2007}. Indeed, if $L$ is a knot, then $\mu_p=0$ for any $p<\infty$, 
and hence $\ds \log |a_0|=\sum_{p<\infty}\sum_{|\alpha_i|_p>1}\log |\alpha_i|_p$ $\ds (=\sum_{p<\infty}h_p)$ holds.

\section{Purely $p$-adic notions} 
\subsection{Besser--Deninger $m_p$}

The notion of the $p$-adic logarithm $\log_p:\C_p^*\to \C_p$ is known (\cite[Chapter 5]{Washington}). 
It is first defined on the unit disc $|z-1|_p<1$ by the power series $\ds \log_p(1+z)=\sum_{n=1}^{\infty}\frac{(-1)^{n+1}z^n}{n}$, and then uniquely extended to whole $\C_p^*$ by the $p$-adic analytic continuation under the normalization $\log_p p=0$. 
It is sometimes called the Iwasawa $p$-adic logarithm. 
Besser and Deninger introduced the \emph{purely} $p$-adic logarithmic Mahler measure $m_p$ for $p$-adic analytic functions (\cite{BesserDeninger1999}). It is different from ours, and is defined by the Shnirel'man integral (\cite{Schnirelmann1938}): 
\begin{dfn} \label{mp}
For a $p$-adic analytic functions $f(z)$ with no zero on $|z|_p=1$, we put 
$$m_p(f(z)):=\int_{|z|_p=1} \log_p f(z)\frac{dz}{z}:=\lim_{n\in \N; {\rm gcd}(n,p)=1} \sum_{\zeta^n=1} \log_pf(z),$$
where the limit is taken with respect to the $p$-adic topology. 
\end{dfn} 
It satisfies Jensen's formula for Laurent polynomials: 
\begin{prp}[Jensen's formula] 
If $f(t)\in \C_p[t^{\pm1}]$ with no root on $|z|_p=1$, then the limit in Definition \ref{mp} exists. 
If $\ds f(t)=\sum_{0\leq i\leq l} a_it^{d-i}=a_0t^{d-l}\prod_{0\leq i\leq l} (t-\alpha_i)$ ($a_i, \alpha_i \in \C_p, a_0a_l\neq 0$), then we have 
\begin{align*} 
m_p(f(t)) &= \log_pa_l -\sum_{0<|\alpha_i|_p<1} \log_p \alpha_i \\
&=\log_p a_0 +\sum_{|\alpha_i|_p>1} \log_p \alpha_i .
\end{align*}
\end{prp} 
Since $\log_p(z)$ does not vanish on $|z|_p=1$, we have no natural generalization of $m_p$ for $f(z)$ with zeros on $|z|_p=1$. 

\begin{rem} In the Shnirel'man integral, we have no term corresponding to $2\pi\sqrt{-1}$ in the complex case.In \cite{Mihara-singular}, its \emph{denormalization} and generalization are given, in which the corresponding terms called periods in $\mathbb{B}_{\rm dR}$ appear. \end{rem} 

By the definition of the Shnirel'man integral, an analogue of the asymptotic formula of resultants (Proposition \ref{asym}, \ref{asym-p}) immediately follows: 

\begin{prp} If $f(z)\in \Cp[t]$ with no root on $|z|_p=1$, then 
$$\lim_{n\in \N; (n,p)=1}\frac{1}{n}\log_p |R(f(t),t^n-1)|=m_p(f(t)).$$
\end{prp}
Since $R(f(t),t^n-1)=R(f(t),\nu_n(t))\times f(1)$, if $f(1)\neq 0$, then 
the convergence of $\{\ds \frac{1}{n} \log_p|R(f(t),\nu_n(t))|\}_n$ is equivalent to 
that of $\{\ds \frac{1}{n}\log_p |f(1)|\}_n$, and hence to $\log_p|f(1)|=0$. \\ 

We apply their theory to the TLN-$\Z$-cover of $S^3$ branched over a $d$-component link $L$. Put $\Lambda=\Z[t^{\pm1}]$. 
We use the notation in the previous section. We have $A_L(t)=(t-1)^dH_L(t)$. 
The Alexander module admits a natural direct sum decomposition $H_1(X_\infty) \cong \Z^{d-1}\oplus H_1'(X_\infty)$ of $\Z[t^\Z]$-modules with ${\rm Fitt}_{\Z[t^\Z]}H_1'(X_\infty)=(H_L(t))$. 

If $d=1$, then $A_L(t)$, $H_L(t)$ and the Alexander polynomial $\Delta_L(t)$ coincide. 
Since $A_L(1)=\pm1$,  by Proposition \ref{Fox}, we have 
$|H_1(M_n)|=|R(A_L(t),\nu_n(t)))|=|R(A_L(t),t^n-1)|$. 
We obtain an analogue of Theorem \ref{pGS} with respect to the $p$-adic norm: 
If $A_L(t)$ has no root on $|z|_p=1$, then 
$$\lim_{n\in \N; (n,p)=1} \frac{1}{n}\log_p|H_1(M_n)|=m_p(A_L(t)).$$

For the cases with $d>1$, note that the trivial action of $t$ on $\Z$ is not expansive and $\hbar_p$ is not defined. It comes from the fact that $\{(\log_p n)/n\}_n$ is not convergent in $\C_p$. 
However, for any $d$, if $H_L(z)$ has no roots on $|z|_p=1$, then the meridian action $t$ on the Hosokawa module $\Lambda/(H_L(t))$ is expansive, and its purely $p$-adic entropy is given by $\hbar_p=m_p(H_L(t))$. 

For a general $d \in \N_{>0}$, we have $|H_1(M_n)|=|R(A_L(t),\nu_n(t)))|=
|R(H_L(t),\nu_n(t))$ $R((t-1)^{d-1},\nu_n(t))|
=|R(H_L(t),t^n-1)/R(H_L(t),t-1)|\times n^{d-1}
=|R(H_L(t),t^n-1)/H_L(1)|\times n^{d-1}.$
Hence a modified asymptotic formula for links is stated as follows: 
\begin{thm} \label{pureGS} 
Let $L$ be a $d$-component link in $S^3$ and 
suppose that the Hosokawa polynomial $H_L(t)$ has no root on $|z|_p=1$. 
Then the TLN-$\Z$-cover over $(S^3,L)$ satisfies 
$$\lim_{n\in \N; (n,p)=1}\frac{1}{n}\log_p \frac{|H_1(M_n)||H_L(1)|}{n^{d-1}}=m_p(H_L(t))$$
with respect to the $p$-adic topology. 
\end{thm}

\subsection{Deninger's $\hbar_p$} 
The topological entropy of a dynamical system (an automorphism $\varphi$ on a space $X$) is sometimes calculated by using the numbers of the fixed points (e.g., \cite[Section 2]{EverestWard1999}): 
$$\ds h(\varphi)=\lim_{n\to \infty} \frac{1}{n}\log|{\rm Fix}(\varphi^n)|.$$ 
Deninger defined \emph{purely} $p$-adic entropy $\hbar_p(\varphi)$ by using the numbers of the fixed points and $p$-adic logarithm under certain situations: 
$$\ds \hbar_p(\varphi):=\lim_{n\to \infty} \frac{1}{n}\log_p |{\rm Fix}(\varphi^n)|.$$ 
A Laurent polynomial $f$ 
has no zeros on the $p$-adic unit torus 
if and only if the associated solenoidal dynamical system is expansive. 
Under those conditions, we have $\hbar_p=m_p(f)$ (\cite[Theorem 1]{Deninger2009}). \\ 

We consider the TLN-$\Z$-cover over a $d$-component link $L$ in $S^3$ again. 
We consider the meridian action on the Alexander module $H_1(X_\infty)\cong \Z^{d-1}\oplus \Lambda/(H_L(t))$. 
Since $\{(\log_p n)/n\}_n$ is not convergent in $\C_p$, 
the trivial action of $t$ on $\Z$ is not expansive. 
Therefore, if $d>1$, then $\hbar_p$ is not defined for $t\act H_1(X_\infty)$. However, we have a formula for a direct summand:   

\begin{prp}
For any $d$, if $H_L(z)$ has no root on $|z|_p=1$, then the meridian action on the Hosokawa module 
$H_1'(X_\infty)$ 
is expansive, and its purely $p$-adic entropy is given by $\hbar_p=m_p(H_L(t))$. 
\end{prp}

\begin{exa} The first example of a knot in Rolfsen's table for which $\hbar_p$ is defined is $L=5_2$ (the 3-twist knot). We have $A_L(t)=\Delta_L(t)=2t^2-3t+2$. Its roots $\alpha,\beta$ satisfy $|\alpha\beta|_2=1$ and $|\alpha+\beta|_2=|3/2|_2=2>1$. Let $\alpha=(3-\sqrt{-7})/4$ denote the larger root. 
Then we have $\hbar_2=m_2(\Delta_L(t))=\log_2(3-\sqrt{-7})$. 

Among the examples of links we have seen, $\hbar_p$ is defined only for the following cases with $p=2$: 

Example \ref{ex2} (1) $L=6^2_2$, $H_L(t)\dot{=}2t^2-t+2$. 
Let $(1-\sqrt{-15})/4$ denote the larger root. Then we have 
$\hbar_2=m_2(A_L(t)/2) = \log_2 (1-\sqrt{-15})$. 

(3) $L=7^2_1$. (ii) $H_L(t)\dot{=}2t^2-3t+2$.  
Let $(3-\sqrt{-7})/4$ denote larger root. Then 
$\hbar_2=\log_2 (3-\sqrt{-7})$. 

(4) $L=7^2_2$, $H_L(t)\dot{=}2t^2-5t+2=2(t-2)(t-\frac{1}{2})$.  
$\hbar_2=0$.  
\end{exa}


\section{Profinite cyclic covers} 
We study entropies of the meridian actions on modules associated to profinite cyclic covers. 

\subsection{$\Zp$-covers} \label{Zp-cover} 
Let $L$ be a link in $S^3$. A \emph{branched $\Zp$-cover} over $(S^3,L)$ is an inverse system $\{M_{p^n}\to S^3\}_n$ of branched $\Z/p^n\Z$-covers branched over $L$. 
If we fix a $\Z$-cover of $S^3-L$, then we obtain a branched $\Zp$-cover from its subcovers. 
A branched $\Zp$-cover is not necessarily obtained from a $\Z$-cover, because 
the image of $\Z^2\to \Zp$ is not necessarily sent to $\Z \subset \Zp$ by automorphisms of $\Zp$. 

We have an isomorphism $\Zp[[t^{\Zp}]]\cong \Zp[[T]]; t\mapsto 1+T$. 
The Iwasawa module (pro-$p$ Alexander module) $\mca{H}_p:=\varprojlim_n H_1(M_{p^n},\Zp)$ of a branched $\Zp$-cover is a finitely generated torsion $\Zp[[T]]$-module (\cite{Ueki2}). 
Let $\tau:H_1(S^3-L)\to \Zp; t_i\mapsto z_i$ ($z_i\in \Zp^\times$) be a homomorphism. 
Then $\Ker$ of the composite of $\tau$ and the natural surjections $\Zp \surj \Z/p^n\Z$ defines an inverse system of $\Z/p^n\Z$-covers, and hence a totally branched $\Zp$-cover associated to $\tau$. 
The Fitting ideal of $\mca{H}_p$ over $\Zp[[t^{\Zp}]]$ is generated by $A_L(t)=(t-1)\Delta_L(t^{z_1},\ldots,t^{z_d})$. 

By the $p$-adic Weierstrass preparation theorem \cite[Theorem 7.3]{Washington}, we have 
$A_L(1+T)=p^{\mu_p}g(T)u(T)$ for a distinguished polynomial $g(T)$ (that is, a monic polynomial with every lower coefficient divisible by $p$) and $u(T)\in \Zp[[T]]^\times$. 
All the roots of $g(T)$ are on $|z|_p<1$ and $\lambda_p=\deg(g(T))$ holds. 
If $u(T)$ has a zero on the domain of convergence, then it is on $|z|_p\geq 1$. 
We have $\ds \mca{H}\otimes \Q \cong \bigoplus_i \Qp[[T]]/(g_i(T))$ and $\ds \prod_i g_i(T)=g(T)$ for some $g_i(T)$'s. Thus we have $\Q_p[[T]]/g_i(T)\cong \Qp^{\deg (g_i)}$ and hence $\mca{H}_p\otimes \Q\cong \Q_p^{\lambda_p}$ as linear spaces over $\Qp$. 
The characteristic polynomial of the meridian action (the $t$ action) on $\mca{H}_p\otimes \Q$ is $g(t-1) \in \Z[t]$, 
and hence its entropy is given by $$h=h_p=\mh_p(g(t-1))=0.$$ 

\subsection{$\wh{\Z}$-covers} 
Next, we consider branched $\wh{\Z}$-covers. 
Note that the profinite integer ring admits the decomposition $\wh{\Z}= \prod_p \Zp$ by the Chinese remainder theorem. 
A \emph{branched $\wh{\Z}$-cover} over $(L,S^3)$ is an inverse system $\{M_n\to S^3\}_n$ of branched $\Z/n\Z$-covers branched over a link. 
Let $\tau: H_1(S^3-L)\to \wh{\Z}; t_i\mapsto z_i$, $z_i\in \wh{\Z}^\times$ be a homomorphism. 
Then the kernels of the composites of $\tau$ and the natural surjections $\wh{\Z}\surj \Z/n\Z$ define an inverse system of branched $\Z/n\Z$-covers, and hence a branched $\wh{\Z}$-cover associated to $\tau$. 
We put $A_L(t):=(t-1)\Delta_L(t^{z_1},\ldots,t^{z_d})$ and write $A_L(t)=(t-1)^\delta H_L(t)$, $H_L(t)\in \Z[[t^{\wh{\Z}}]]$, $H_L(1)\neq 0$. 
There is a natural isomorphism between $H_1(M_n;\wh{\Z})$ and the profinite completion $\wh{H}_1(M_n)$ of $H_1(M_n)$. There is a natural direct sum decomposition ${H}_1(M_n)\congto \Z/n\Z^{\oplus \delta}\oplus \wh{H}_1'(M_n)$ with $\delta\in \N$ such that the Fitting ideal of $\wh{H}_1'(M_n)$ over $\wh{\Z}[t^{\Z/n\Z}]$ is given by 
${\rm Fitt}_{\Z[t^{\Z/n\Z}]}\wh{H}_1'(M_n)=(H_L(t))\subset \wh{\Z}[t^{\Z/n\Z}].$ 

We define the completed Alexander module by putting $\mca{H}:=\varprojlim_n H_1(M_n,\wh{\Z})$. 
This module is a finitely generated $\wh{\Z}[[t^{\wh{\Z}}]]$-module. Indeed, let $z_i=(z_{i,n}\mod n)_n$ with $z_{i,n}\in \Z$ for each $i$. Let $h_{\infty,n}$ denote the $\wh{\Z}$-cover corresponding to the map $\tau_n: H_1(S^3-L)\mapsto \wh{\Z}; t_i\mapsto z_{i,n}$ for each $n$. Then the completed Alexander module $\mca{H}_n$ of $h_{\infty,n}$ is a finitely generated $\wh{\Z}[[t^{\Z}]]$-module by a similar argument to \cite[Lemma 11]{Ueki5}. Since the ranks of $\mca{H}_n$'s are bounded and $\mca{H}$ is approximated by $\mca{H}_n$'s, $\mca{H}$ is a finitely generated $\wh{\Z}[[t^{\Z}]]$-module. 
We have a natural direct sum decomposition $\mca{H}\cong \wh{\Z}^{\oplus \delta} \oplus \mca{H}'$ such that the Fitting ideal of $\mca{H}'$ is $(H_L(t))\subset \wh{\Z}[[t^{\Z}]]$. 

A branched $\wh{\Z}$-cover which cannot be obtained from a $\Z$-cover may admit $A_L(t)$ in $\Z[t]$: 

\begin{exa} 
Let $L=K_1\cup K_2\cup K_3 \subset S^3$ be a 3-component link with $\Delta_L(x,y,z)=B(xyz)\in \Z[xyz]$. 
Take $u\in \wh{\Z}^\times$ with $u\neq \pm1$. Let $t_i$ denote the meridian of $K_i$ in $H_1(S^3-L)$ for each $i$. Consider tha map $\tau:H_1(S^3-L)\to \wh{\Z}; t_1\mapsto 1, t_2\mapsto u, t_3\mapsto -u$. Then it defines a branched $\wh{\Z}$-cover which cannot be obtained from a $\Z$-cover and admits the characteristic polynomial $A_L(t)=(1-t)\Delta_L(t,t^u,t^{-u})=(1-t)B(t)\in \Z[t]$. 

There is only one non-trivial example in the Rolfsen table (\cite{Rolfsen1976}): 
the link $L=8^3_7$ satisfies $\Delta_L(x,y,z)=1-xyz$, $A_L(t)=(1-t)\Delta_L(t^u,t^{-u},t)=(1-t)^2$. 
\end{exa}

A $\wh{\Z}$-cover is approximated by $\Z$-covers, that is, every layer can be obtained as a quotient of a $\Z$-cover. Hence a generalization $$|\wh{H}_1(M_n)|=|R(A_L(t),\nu_n(t))|$$ of Fox's formula (Proposition \ref{Fox}) holds. 
Therefore, if $A_L(t) \in \Z[t]$, then for any $p\leq \infty$, generalizations 
of asymptotic formulae with use of $p$-adic Mahler measures (Theorems \ref{pGS}) hold: 
$$\ds \lim_{{\rm value}\neq 0}|H_1(M_n)|_p=\mh_p(A_L(t)).$$ 
In addition, if $H_L(t):=A_L(t)/(t-1)^{d-1}$ has no root on roots of unity, then 
a generalization of the purely $p$-adic version (Theorem \ref{pureGS}) for each $p<\infty$ holds: 
$$\ds \lim_{n\in \N; (n,p)=1}\frac{1}{n}\log_p \frac{|H_1(M_n)|}{n^{d-1}}=m_p(H_L(t)).$$

In the following, we consider a branched $\wh{\Z}$-cover with $A_L(t) \in \Z[t]$ and study the entropy of the meridian action ($t$-action). 
Note the decomposition $\wh{\Z}[[t^{\wh{\Z}}]]\cong \prod_p \Zp[[t^{\wh{\Z}}]]$. 
Let $\restprod_p \Qp$ denotes the restricted product of $\Qp$'s with respect to open subgroups $\Zp<\Qp$. 

Here we assume $\mca{H}\otimes \Q\cong (\restprod_p \Qp)^{{\rm deg}(A_L(t))}$. 
Then for each prime number $p$, the degree of $A_L(t) \mod p$ in $\Fp[t]$ coincides with that of $A_L(t)\in \Z[t]$. Indeed, since $\mca{H}$ is a quotient of compact Hausdorff group, it has no $p$-divisible factor. 
Hence the dimension of $\mca{H}\otimes \Q$ over $\Qp$ and that of $\mca{H}/p\mca{H}$ over $\Fp$ must be coincide. 
Therefore, the leading coefficient of $A_L(t)$ is not divisible by any $p$. 
Since the characteristic polynomial of the $t$-action on each $\Qp^{{\rm deg}(A_L(t))}$ is $A_L(t)$, 
its $p$-adic entropy is given by $h_p=\log \mh_p(A_L(t))=0$. 

The pontryagin dual of $\restprod_p \Qp$ is isomorphic to $\restprod_p\Qp$. 
We can replace the ad\`ele ring $\bb{A}_\Q$ by $\restprod_p\Qp$ in the argument of \cite[Lemmas 4.3-4.5]{LindWard1988} with use of results of \cite{Bowen1971tams-e} and \cite{WaltersSpringer1982}. 
Since the finite set $P$ of prime numbers with contribution is empty, 
the topological entropy $h$ of the $t$-action on the dual of $\mca{H}_\Q$ is given by $h=0$.

\subsection*{Acknowledgment}
I would like to express my sincere gratitude to Mikio Furuta and Yuichiro Taguchi for their precise questions at the thesis defense. 
I am grateful to Abhijit Champanerkar, Teruhisa Kadokami, Matilde Lal\'{i}n, Yoshihiko Matsumoto, Yasushi Mizusawa, Takayuki Morisawa, Masanori Morishita,  Hirofumi Niibo, Yuji Terashima, and people at the coffee time in Komaba for helpful comments. 
I would like to thank Tomoki Mihara for precious guidance to $p$-adic analysis. 
I am also grateful to Dohyeong Kim for inviting me to the beautiful environment of POSTECH IBS-CBP, 
Noburu Ito, Kouki Taniyama, and Seiken Saito for serving important opportunities to give talks, 
my family and friends for essential support. 
I was partially supported by Grant-in-Aid for JSPS Fellows (25-2241). 



\bibliographystyle{amsalpha}
\bibliography{/Users/uekijun/Dropbox/refs1}

\providecommand{\bysame}{\leavevmode\hbox to3em{\hrulefill}\thinspace}
\providecommand{\MR}{\relax\ifhmode\unskip\space\fi MR }
\providecommand{\MRhref}[2]{%
  \href{http://www.ams.org/mathscinet-getitem?mr=#1}{#2}
}
\providecommand{\href}[2]{#2}
\begin{thebibliography}{BRVD03}

\bibitem[AKM65]{AKM1965}
R.~L. Adler, A.~G. Konheim, and M.~H. McAndrew, \emph{Topological entropy},
  Trans. Amer. Math. Soc. \textbf{114} (1965), 309--319. \MR{0175106 (30
  \#5291)}

\bibitem[BD99]{BesserDeninger1999}
Amnon Besser and Christopher Deninger, \emph{{$p$}-adic {M}ahler measures}, J.
  Reine Angew. Math. \textbf{517} (1999), 19--50. \MR{1728549 (2001d:11070)}

\bibitem[Bow71]{Bowen1971tams-e}
Rufus Bowen, \emph{Entropy for group endomorphisms and homogeneous spaces},
  Trans. Amer. Math. Soc. \textbf{153} (1971), 401--414. \MR{0274707}

\bibitem[Boy02]{Boyd2002Mahler}
David~W. Boyd, \emph{Mahler's measure and invariants of hyperbolic manifolds},
  Number theory for the millennium, {I} ({U}rbana, {IL}, 2000), A K Peters,
  Natick, MA, 2002, pp.~127--143. \MR{1956222}

\bibitem[BRVD03]{BoydRVD2003}
David~W. Boyd, Fernando Rodriguez-Villegas, and Natan Dunfield, \emph{Mahler's
  measure and the dilogarithm ({II})}, Preprint, arXiv:0308041v2, 2003.

\bibitem[BV13]{BV2013}
Nicolas Bergeron and Akshay Venkatesh, \emph{The asymptotic growth of torsion
  homology for arithmetic groups}, J. Inst. Math. Jussieu \textbf{12} (2013),
  no.~2, 391--447. \MR{3028790}

\bibitem[Den09]{Deninger2009}
Christopher Deninger, \emph{{$p$}-adic entropy and a {$p$}-adic
  {F}uglede-{K}adison determinant}, Algebra, arithmetic, and geometry: in honor
  of {Y}u. {I}. {M}anin. {V}ol. {I}, Progr. Math., vol. 269, Birkh\"auser
  Boston, Inc., Boston, MA, 2009, pp.~423--442. \MR{2641178 (2011k:37143)}

\bibitem[DGB14]{DikranjanBruno2014}
Dikran Dikranjan and Anna Giordano~Bruno, \emph{The bridge theorem for totally
  disconnected {LCA} groups}, Topology Appl. \textbf{169} (2014), 21--32.
  \MR{3199856}

\bibitem[EW99]{EverestWard1999}
Graham Everest and Thomas Ward, \emph{Heights of polynomials and entropy in
  algebraic dynamics}, Universitext, Springer-Verlag London, Ltd., London,
  1999. \MR{1700272 (2000e:11087)}

\bibitem[Fox56]{FoxFDC3}
Ralph~H. Fox, \emph{Free differential calculus. {III}. {S}ubgroups}, Ann. of
  Math. (2) \textbf{64} (1956), 407--419. \MR{0095876}

\bibitem[GAS91]{GS1991}
Francisco Gonz{\'a}lez-Acu{\~n}a and Hamish Short, \emph{Cyclic branched
  coverings of knots and homology spheres}, Rev. Mat. Univ. Complut. Madrid
  \textbf{4} (1991), no.~1, 97--120. \MR{1142552 (93g:57004)}

\bibitem[GBV15]{BrunoVirili2015}
Anna Giordano~Bruno and Simone Virili, \emph{Algebraic {Y}uzvinski formula}, J.
  Algebra \textbf{423} (2015), 114--147. \MR{3283712}

\bibitem[Hil12]{Hillman2}
Jonathan Hillman, \emph{Algebraic invariants of links}, second ed., Series on
  Knots and Everything, vol.~52, World Scientific Publishing Co. Pte. Ltd.,
  Hackensack, NJ, 2012. \MR{2931688}

\bibitem[HMM06]{HMM2006}
Jonathan Hillman, Daniel Matei, and Masanori Morishita, \emph{Pro-{$p$} link
  groups and {$p$}-homology groups}, Primes and knots, Contemp. Math., vol.
  416, Amer. Math. Soc., Providence, RI, 2006, pp.~121--136. \MR{2276139
  (2008m:57010)}

\bibitem[Iwa59]{Iwasawa1959}
Kenkichi Iwasawa, \emph{On {$\Gamma $}-extensions of algebraic number fields},
  Bull. Amer. Math. Soc. \textbf{65} (1959), 183--226. \MR{0124316 (23
  \#A1630)}

\bibitem[KM08]{KM2008}
Teruhisa Kadokami and Yasushi Mizusawa, \emph{Iwasawa type formula for covers
  of a link in a rational homology sphere}, J. Knot Theory Ramifications
  \textbf{17} (2008), no.~10, 1199--1221. \MR{2460171 (2009g:57023)}

\bibitem[KM13]{KM2013}
\bysame, \emph{On the {I}wasawa invariants of a link in the 3-sphere}, Kyushu
  J. Math. \textbf{67} (2013), no.~1, 215--226. \MR{3089003}

\bibitem[Lal03]{Lalin2003}
Matilde~N. Lal\'in, \emph{Some examples of mahler measures as multiple
  polylogarithms}, J. Number Theory \textbf{103} (2003), no.~1, 85--108.

\bibitem[Lal04]{Lalin2004}
\bysame, \emph{Mahler measure and volumes in hyperbolic space}, Geom. Dedicata
  \textbf{107} (2004), 211--234. \MR{2110763}

\bibitem[Le14]{Le2014}
Thang Le, \emph{Homology torsion growth and {M}ahler measure}, Comment. Math.
  Helv. \textbf{89} (2014), no.~3, 719--757. \MR{3260847}

\bibitem[LW88]{LindWard1988}
D.~A. Lind and T.~Ward, \emph{Automorphisms of solenoids and {$p$}-adic
  entropy}, Ergodic Theory Dynam. Systems \textbf{8} (1988), no.~3, 411--419.
  \MR{961739 (90a:28031)}

\bibitem[Mih12]{Mihara-singular}
Tomoki Mihara, \emph{Singular homology of non-archimedean analytic spaces and
  integration along cycles}, preprint. arXiv:1211.1422v1, 2012.

\bibitem[MM82]{MM1982}
John~P. Mayberry and Kunio Murasugi, \emph{Torsion-groups of abelian coverings
  of links}, Trans. Amer. Math. Soc. \textbf{271} (1982), no.~1, 143--173.
  \MR{648083 (84d:57004)}

\bibitem[Nog07]{Noguchi2007}
Akio Noguchi, \emph{Zeros of the {A}lexander polynomial of knot}, Osaka J.
  Math. \textbf{44} (2007), no.~3, 567--577. \MR{2360941 (2009f:57020)}

\bibitem[Por04]{Porti2004}
Joan Porti, \emph{Mayberry-{M}urasugi's formula for links in homology
  3-spheres}, Proc. Amer. Math. Soc. \textbf{132} (2004), no.~11, 3423--3431
  (electronic). \MR{2073320 (2005e:57004)}

\bibitem[Ril90]{Riley1990gr}
Robert Riley, \emph{Growth of order of homology of cyclic branched covers of
  knots}, Bull. London Math. Soc. \textbf{22} (1990), no.~3, 287--297.
  \MR{1041145 (92g:57017)}

\bibitem[Rol76]{Rolfsen1976}
Dale Rolfsen, \emph{Knots and links}, Publish or Perish, Inc., Berkeley,
  Calif., 1976, Mathematics Lecture Series, No. 7. \MR{0515288 (58 \#24236)}

\bibitem[Sak81]{Sakuma1981}
Makoto Sakuma, \emph{On the polynomials of periodic links}, Math. Ann.
  \textbf{257} (1981), no.~4, 487--494. \MR{639581 (83i:57003)}

\bibitem[Shn38]{Schnirelmann1938}
Lev~Genrikhovich Shnirel'man(Schnirelmann), \emph{Sur les fonctions dans les
  corps norm\'es et alg\'ebriquement ferm\'es}, Izv. Akad. Nauk SSSR Ser. Mat.
  \textbf{2} (1938), no.~5-6, 487--498 (Russian).

\bibitem[Smy81]{Smyth1981}
C.~J. Smyth, \emph{On measures of polynomials in several variables}, Bull.
  Austral. Math. Soc. \textbf{23} (1981), no.~1, 49--63. \MR{615132}

\bibitem[SW02]{SW2002M}
Daniel~S. Silver and Susan~G. Williams, \emph{Mahler measure, links and
  homology growth}, Topology \textbf{41} (2002), no.~5, 979--991. \MR{1923995
  (2003h:57011)}

\bibitem[Tan17]{TangeRyoto1}
Ryoto Tange, \emph{{F}ox formulas for twisted {A}lexander invariants associated
  to representations of knot groups over rings of {$S$}-integers}, preprint,
  arXiv:, 2017.

\bibitem[Uek16]{Ueki3}
Jun Ueki, \emph{On the {I}wasawa {$\mu$}-invariants of branched
  {$\bold{Z}_p$}-covers}, Proc. Japan Acad. Ser. A Math. Sci. \textbf{92}
  (2016), no.~6, 67--72. \MR{3508576}

\bibitem[Uek17]{Ueki2}
\bysame, \emph{On the {I}wasawa invariants for links and {K}ida's formula},
  Internat. J. Math. \textbf{28} (2017), no.~6, 1750035, 30. \MR{3663789}

\bibitem[Uek18]{Ueki5}
\bysame, \emph{The profinite completions of knot groups determine the
  {A}lexander polynomials}, to appear in Algebr. Geom. Topol. (2018), 13 pages,
  (arXiv:1702.03836).

\bibitem[Wal82]{WaltersSpringer1982}
Peter Walters, \emph{An introduction to ergodic theory}, Graduate Texts in
  Mathematics, vol.~79, Springer-Verlag, New York-Berlin, 1982. \MR{648108}

\bibitem[Was97]{Washington}
Lawrence~C. Washington, \emph{Introduction to cyclotomic fields}, second ed.,
  Graduate Texts in Mathematics, vol.~83, Springer-Verlag, New York, 1997.
  \MR{1421575 (97h:11130)}

\bibitem[Web79]{Weber1979}
Claude Weber, \emph{Sur une formule de {R}. {H}. {F}ox concernant l'homologie
  des rev\^etements cycliques}, Enseign. Math. (2) \textbf{25} (1979), no.~3-4,
  261--272 (1980). \MR{570312 (81d:57011)}

\end{thebibliography}


\ \\
Jun Ueki \\
Graduate School of Mathematical Sciences, The University of Tokyo, \\
3-8-1 Komaba, Meguro-ku, Tokyo, 153-8914, Japan\\
E-mail: \url{uekijun46@gmail.com} 


\end{document}